\documentclass[11pt,reqno]{article}

\usepackage[usenames]{color}
\usepackage{amssymb}
\usepackage{amsmath}
\usepackage{amsthm}
\usepackage{amsfonts}
\usepackage{amscd}
\usepackage{graphicx}

\usepackage[colorlinks=true,
linkcolor=webgreen,
filecolor=webbrown,
citecolor=webgreen]{hyperref}

\definecolor{webgreen}{rgb}{0,.5,0}
\definecolor{webbrown}{rgb}{.6,0,0}
\usepackage{color}
\usepackage{fullpage}
\usepackage{float}

\usepackage{cases}

\usepackage{graphics,amsmath,amssymb}
\usepackage{amsthm}
\usepackage{amsfonts}
\usepackage{latexsym}

\begin{document}

\theoremstyle{plain}
\newtheorem{theorem}{Theorem}
\newtheorem{remark}{Remark}
\newtheorem{lemma}{Lemma}
\newtheorem{definition}{Definition}
\newtheorem{proposition}{Proposition}
\newtheorem{corollary}{Corollary}

\begin{center}
\vskip 1cm{\Large\bf Sums of powers of integers and generalized\\
\vskip .08in Stirling numbers of the second kind}
\vskip .2in \large Jos\'{e} Luis Cereceda \\
{\normalsize Collado Villalba, 28400 (Madrid), Spain} \\
\href{mailto:jl.cereceda@movistar.es}{\normalsize{\tt jl.cereceda@movistar.es}}
\end{center}

\begin{abstract}
By applying the Newton-Gregory expansion to the polynomial associated with the sum of powers of integers $S_k(n) = 1^k + 2^k + \cdots + n^k$, we derive a couple of infinite families of explicit formulas for $S_k(n)$. One of the families involves the $r$-Stirling numbers of the second kind $\genfrac{\{}{\}}{0pt}{}{k}{j}_r$, $j=0,1,\ldots,k$, while the other involves their duals $\genfrac{\{}{\}}{0pt}{}{k}{j}_{-r}$, with both families of formulas being indexed by the non-negative integer $r$. As a by-product, we obtain three additional formulas for $S_k(n)$ involving the numbers $\genfrac{\{}{\}}{0pt}{}{k}{j}_{n+m}$, $\genfrac{\{}{\}}{0pt}{}{k}{j}_{n-m}$ (where $m$ is any given non-negative integer), and $\genfrac{\{}{\}}{0pt}{}{k}{j}_{k-j}$, respectively. Moreover, we provide a formula for the Bernoulli polynomials $B_k(x-1)$ in terms of $\genfrac{\{}{\}}{0pt}{}{k}{j}_{x}$ and the harmonic numbers.
\end{abstract}

\section{Introduction}

Following Broder \cite[Equation 57]{broder} (see also Carlitz \cite[Equation (3.2)]{carlitz}) we define the generalized (or weighted) Stirling numbers of the second kind by
\begin{equation*}\label{rsp}
R_{k,j}(x) = \sum_{i=0}^{k-j} \binom{k}{i} \genfrac{\{}{\}}{0pt}{}{k-i}{j} x^i, \quad \text{integers}\,\, 0 \leq j \leq k,
\end{equation*}
where $x$ stands for any arbitrary real or complex value, and where the $\genfrac{\{}{\}}{0pt}{}{k}{j}$'s are the ordinary Stirling numbers of the second kind \cite{boya}. Note that $R_{k,j}(x)$ is a polynomial in $x$ of degree $k-j$ with leading coefficient $\binom{k}{j}$ and constant term $R_{k,j}(0) = \genfrac{\{}{\}}{0pt}{}{k}{j}$. Furthermore, we have that $R_{k,j}(1) = \genfrac{\{}{\}}{0pt}{}{k+1}{j+1}$. In general, when $x$ is the non-negative integer $r$, $R_{k,j}(r)$ becomes the $r$-Stirling number of the second kind $\genfrac{\{}{\}} {0pt}{}{k+r}{j+r}_r$ \cite{broder}. A combinatorial interpretation of the polynomial $R_{k,j}(x)$ is given in \cite[Theorem 27]{broder} (see also the definition provided by B\'{e}nyi and Matsusaka in \cite[Definition 2.13]{matsu}).

For convenience and notational simplicity, in this paper we employ the notation $\genfrac{\{}{\}} {0pt}{}{k}{j}_r$ to refer to Broder's $r$-Stirling numbers of the second kind $\genfrac{\{}{\}} {0pt}{}{k+r}{j+r}_r$. The former notation has been used recently by Ma and Wang in \cite{ma} (see also \cite{matsu} and \cite{merris}). The numbers $\genfrac{\{}{\}} {0pt}{}{k}{j}_r$ are then given by
\begin{equation*}
\genfrac{\{}{\}}{0pt}{}{k}{j}_r = \sum_{i=0}^{k-j} \binom{k}{i} \genfrac{\{}{\}}{0pt}{}{k-i}{j} r^i,
\quad \text{integer} \,\, r \geq 0.
\end{equation*}
Likewise, adopting the notation in \cite{ma}, we define the counterpart or dual of $\genfrac{\{}{\}} {0pt}{}{k}{j}_r$ for negative integer $r$ as
\begin{equation*}
\genfrac{\{}{\}}{0pt}{}{k}{j}_{-r} = \sum_{i=0}^{k-j} (-1)^i \binom{k}{i} \genfrac{\{}{\}}{0pt}{}{k-i}{j} r^i,
\quad \text{integer} \,\, r \geq 0.
\end{equation*}
Alternatively, $\genfrac{\{}{\}} {0pt}{}{k}{j}_r$ and $\genfrac{\{}{\}} {0pt}{}{k}{j}_{-r}$ can be equivalently expressed in the form
\begin{align}
\genfrac{\{}{\}}{0pt}{}{k}{j}_r & = \frac{1}{j!} \sum_{i=0}^{j} (-1)^{j-i} \binom{j}{i} (i+r)^k,
\quad \text{integer} \,\, r \geq 0, \label{rsn} \\
\intertext{and}
\genfrac{\{}{\}}{0pt}{}{k}{j}_{-r} & = \frac{1}{j!} \sum_{i=0}^{j} (-1)^{j-i} \binom{j}{i} (i-r)^k,
\quad \text{integer} \,\, r \geq 0,  \label{dsn}
\end{align}
respectively. Clearly, both $\genfrac{\{}{\}}{0pt}{}{k}{j}_r$ and $\genfrac{\{}{\}}{0pt}{}{k}{j}_{-r}$ reduce to $\genfrac{\{}{\}}{0pt}{}{k}{j}$ when $r=0$. It is to be noted that the numbers $\genfrac{\{}{\}}{0pt}{}{k}{j}_{-r}$ were introduced and studied by Koutras under the name of non-central Stirling numbers of the second kind and denoted by $S_r(k,j)$ (see \cite[Equations (2.5) and (2.6)]{koutras}).

On the other hand, for non-negative integer $k$, let $S_k(n)$ denote the sum of $k$-th powers of the first $n$ positive integers
\begin{equation*}
S_k(n) = 1^k + 2^k + \cdots + n^k,
\end{equation*}
with $S_k(0) =0$ for all $k$. As is well known, $S_k(n)$ can be expressed in terms of the Stirling numbers of the second kind as (see, e.g., \cite{shirali})
\begin{equation}\label{f1}
S_k(n) = -\delta_{k,0} + \sum_{j=0}^{k} j! \binom{n+1}{j+1} \genfrac{\{}{\}}{0pt}{}{k}{j},
\end{equation}
where $\delta_{k,0}$ is the Kronecker delta, which ensures that $S_0(n) = n$. Additionally, $S_k(n)$ admits the following variant of \eqref{f1}:
\begin{equation}\label{f2}
S_k(n) = \sum_{j=1}^{k+1} (j-1)! \binom{n}{j} \genfrac{\{}{\}}{0pt}{}{k+1}{j} = \sum_{j=0}^{k} j! \binom{n}{j+1} \genfrac{\{}{\}}{0pt}{}{k+1}{j+1},
\end{equation}
(see, e.g., \cite[Equation (9)]{witula}, \cite{cere}, \cite[Corollary 2]{kargin}, and \cite[Theorem 5]{chrysafi}). The first expression in \eqref{f2} can be readily obtained from the exponential generating function \cite[Equation (11)]{boya2}
\begin{equation*}
\sum_{n=1}^{\infty} (1^k + 2^k + \cdots + n^k) \frac{x^n}{n!} = e^x \sum_{j=1}^{k+1} \frac{1}{j}
\genfrac{\{}{\}}{0pt}{}{k+1}{j} x^j.
\end{equation*}
Of course, \eqref{f1} and \eqref{f2} are equivalent formulas. Indeed, it is a simple exercise to convert \eqref{f1} into \eqref{f2}, and vice versa, by means of the recursion $\genfrac{\{}{\}}{0pt}{}{k}{j} = j \genfrac{\{}{\}}{0pt}{}{k-1}{j} + \genfrac{\{}{\}}{0pt}{}{k-1}{j-1}$ and the well-known combinatorial identity $\binom{n}{j+1} + \binom{n}{j} = \binom{n+1}{j+1}$.

Incidentally, it is worthwhile to mention that, in his 1928 Monthly article \cite{ginsburg}, Ginsburg wrote down explicitly the first few instances of \eqref{f2} for $k =2,3,4,5$ in terms of the binomial coefficients $\binom{n}{j+1}$, where $j =0,1,\ldots,k$, namely
\begin{align*}
S_2(n) & = \binom{n}{1} + 3\binom{n}{2} + 2\binom{n}{3}, \\
S_3(n) & = \binom{n}{1} + 7\binom{n}{2} + 12\binom{n}{3} + 6\binom{n}{4}, \\
S_4(n) & = \binom{n}{1} + 15\binom{n}{2} + 50\binom{n}{3} + 60\binom{n}{4} + 24\binom{n}{5}, \\
S_5(n) & = \binom{n}{1} + 31\binom{n}{2} + 180\binom{n}{3} + 390\binom{n}{4} + 360\binom{n}{5} + 120\binom{n}{6}.
\end{align*}
As noted by Ginsburg, the above formulas appeared on page 88 of the book by Schwatt, {\it Introduction to Operations with Series\/} (Philadelphia, The Press of the University of Pennsylvania, 1924).

In this paper, we obtain a unifying formula for $S_k(n)$ giving \eqref{f1} and \eqref{f2} as particular cases. Indeed, we derive a couple of infinite families of explicit formulas for $S_k(n)$, one of them involving the numbers $\genfrac{\{}{\}}{0pt}{}{k}{j}_r$ and the other the numbers $\genfrac{\{}{\}}{0pt}{}{k}{j}_{-r}$, with $j=0,1,\ldots,k$. Specifically, we establish the following theorem which constitutes the main result of this paper.

\begin{theorem}\label{th:1}
Let $k$ and $n$ be any non-negative integers and let $\genfrac{\{}{\}}{0pt}{}{k}{j}_r$ and $\genfrac{\{}{\}}{0pt}{}{k}{j}_{-r}$ be the numbers defined in \eqref{rsn} and \eqref{dsn}, respectively, where $r$ stands for any arbitrary but fixed non-negative integer. Then
\begin{align}
S_k(n) & = \sum_{j=0}^{k} j! \left[\binom{n+1-r}{j+1} + (-1)^j \binom{r+j-1}{j+1} \right]
\genfrac{\{}{\}}{0pt}{}{k}{j}_r,  \label{th1}  \\[-4mm]
\intertext{and}
S_k(n) & = \sum_{j=0}^{k} j! \left[\binom{n+1+r}{j+1} - \binom{r+1}{j+1} \right]
\genfrac{\{}{\}}{0pt}{}{k}{j}_{-r}.  \label{th2}
\end{align}
\end{theorem}

Before we prove Theorem \ref{th:1} in the next section, a few observations are in order.

\begin{remark}
It is easily seen that both \eqref{th1} and \eqref{th2} reduce to \eqref{f1} when $r=0$. Furthermore, \eqref{th1} reduces to \eqref{f2} when $r=1$. Moreover, setting $r =n$ in \eqref{th1} leads to
\begin{equation}\label{reqn}
S_k(n) = n^{k+1} + \sum_{j=1}^{k} (-1)^j j! \binom{n+j-1}{j+1} \genfrac{\{}{\}}{0pt}{}{k}{j}_{n}.
\end{equation}
Similarly, setting $r=n+1$ in \eqref{th1} yields
\begin{equation}\label{reqn1}
S_k(n) = \sum_{j=0}^{k} (-1)^j j! \binom{n+j}{j+1} \genfrac{\{}{\}}{0pt}{}{k}{j}_{n+1},
\end{equation}
retrieving the result obtained in \cite[Equation (4.8)]{kargin2}.
\end{remark}

\begin{remark}
It should be stressed that both \eqref{th1} and \eqref{th2} hold irrespective of the value taken by the non-negative integer parameter $r$. This means that, actually, the right-hand side of \eqref{th1} and \eqref{th2} provides us with an infinite supply of explicit formulas for $S_k(n)$, one for each value of $r$. For example, for $r =2$, and noting that $S_k(1) =1$ for all $k$, we have from \eqref{th1}
\begin{equation*}
S_k(n) = 1 + \sum_{j=0}^{k} j! \binom{n-1}{j+1} \genfrac{\{}{\}}{0pt}{}{k}{j}_{2},
\end{equation*}
where
\begin{equation*}
\genfrac{\{}{\}}{0pt}{}{k}{j}_{2} = \frac{1}{j!} \sum_{i=0}^{j} (-1)^{j-i} \binom{j}{i} (i+2)^k.
\end{equation*}
Analogously, for $r=2$, we have from \eqref{th2}
\begin{equation*}
S_k(n) = -\delta_{k,0} + (-1)^{k+1} (1+ 2^k ) + \sum_{j=0}^{k} j! \binom{n+3}{j+1} \genfrac{\{}{\}}{0pt}{}{k}{j}_{-2},
\end{equation*}
where
\begin{equation*}
\genfrac{\{}{\}}{0pt}{}{k}{j}_{-2} = \frac{1}{j!} \sum_{i=0}^{j} (-1)^{j-i} \binom{j}{i} (i-2)^k.
\end{equation*}
\end{remark}

\begin{remark}
In the last section, we obtain a more general formula for $S_k(n)$ involving the numbers $\genfrac{\{}{\}}{0pt}{}{k}{j}_{n+m}$ and $\genfrac{\{}{\}}{0pt}{}{k}{j}_{n-m}$, where $m$ is any given non-negative integer (see equations \eqref{con1} and \eqref{con2}). Furthermore, we provide another formula for $S_k(n)$ involving the numbers $\genfrac{\{}{\}}{0pt}{}{k}{j}_{k-j}$ (see equation \eqref{con3}), as well as a formula for the Bernoulli polynomials $B_k(x-1)$ in terms of $\genfrac{\{}{\}}{0pt}{}{k}{j}_{x}$ and the harmonic numbers (see equation \eqref{ber1}).
\end{remark}

\section{Proof of Theorem \ref{th:1}}

The proof of Theorem \ref{th:1} is based on the following lemma.

\begin{lemma}
For $x$ a real or complex variable, let $S_k(x)$ denote the unique interpolating polynomial in $x$ of degree $k+1$ such that $S_k(x) = 1^k + 2^k + \cdots + x^k$ whenever $x$ is a positive integer (with $S_k(0) =0$). Then
\begin{equation}\label{lem1}
S_k(x) = S_k(a-1) + \sum_{j=0}^{k} j! \binom{x+1-a}{j+1} R_{k,j}(a),
\end{equation}
where $a$ is a parameter taking any arbitrary but fixed real or complex value.
\end{lemma}
\begin{proof}
As is well known (see, e.g., \cite[Equation (15]{atlan}), $S_k(x)$ can be expressed in terms of the Bernoulli polynomials $B_k(x)$ as follows
\begin{equation}\label{pf1}
S_k(x) = \frac{1}{k+1} [B_{k+1}(x+1) - B_{k+1}(1)].
\end{equation}
Recall further that the forward difference operator $\Delta$ acting on the function $f(x)$ is defined by $\Delta f(x) = f(x+1) - f(x)$. Then, the following elementary result
\begin{equation}\label{lem2}
\Delta S_k(x) = (x+1)^k
\end{equation}
follows immediately from \eqref{pf1} and the difference equation $\Delta B_{k+1}(x) = (k+1) x^k$ \cite[Equation (12)]{atlan}.

On the other hand, the Newton-Gregory expansion of the function $f(x)$ is given by (see, e.g., \cite[Equation (A.9), p. 230]{gould})
\begin{equation*}
f(x) = \sum_{j=0}^{\infty} \binom{x-a}{j} \Delta^j f(a),
\end{equation*}
where, for any integer $j \geq 1$, the $j$-th order difference operator $\Delta^j$ is defined by $\Delta^j f(x) = \Delta (\Delta^{j-1} f(x)) = \Delta^{j-1} (\Delta f(x))$ and $\Delta^0 f(x) = f(x)$, and where $\Delta^j f(a) = \Delta^j f(x)|_{x=a}$. Hence, applying the Newton-Gregory expansion to the power sum polynomial $S_k(x)$ and using \eqref{lem2} yields
\begin{equation}\label{pf2}
S_k(x) = S_k(a) + \sum_{j=0}^{k} \binom{x-a}{j+1} \Delta^j (a+1)^k,
\end{equation}
where the terms in the summation with index $j$ greater than $k$ have been omitted because \linebreak $\Delta^j (x+1)^k =0$ for all $j \geq k+1$ \cite[Equation (6.16), p. 68]{gould}.

The connection between \eqref{pf2} and the generalized Stirling numbers $R_{k,j}(x)$ stems from the fact that (see, e.g., \cite[Theorem 29]{broder} and \cite[Equation (3.8)]{carlitz})
\begin{equation}\label{pf3}
R_{k,j}(x) = \frac{1}{j!} \Delta^j x^k.
\end{equation}
Therefore, combining \eqref{pf2} and \eqref{pf3}, and making $a \to a-1$, we get \eqref{lem1}.
\end{proof}

When $x$ and $a$ are the non-negative integers $n$ and $r$, respectively, \eqref{lem1} becomes
\begin{equation}\label{rah1}
S_k(n) = S_k(r-1) + \sum_{j=0}^{k} j! \binom{n+1-r}{j+1} \genfrac{\{}{\}}{0pt}{}{k}{j}_r ,
\end{equation}
with $S_k(-1) =0$ for all $k \geq 1$, and $S_0(-1) = -1$. Now, letting $n=0$ in \eqref{rah1} and using the relation
\begin{equation}\label{relt}
\binom{-x}{k} = (-1)^{k} \binom{x+k-1}{k}
\end{equation}
gives
\begin{equation}\label{rminus}
S_k(r-1) = \sum_{j=0}^{k} (-1)^j j! \binom{r+j-1}{j+1} \genfrac{\{}{\}}{0pt}{}{k}{j}_r.
\end{equation}
Then, substituting \eqref{rminus} in \eqref{rah1}, we get \eqref{th1}.

Moreover, making the transformation $r \to -r$ in \eqref{rah1} and invoking the symmetry property of the power sum polynomial (see, e.g., \cite[Theorem 10]{newsome})
\begin{equation*}
S_k(-r-1) = -\delta_{k,0} + (-1)^{k+1} S_k(r),
\end{equation*}
it follows that
\begin{equation}\label{rah2}
S_k(n) = -\delta_{k,0} - (-1)^{k} S_k(r) + \sum_{j=0}^{k} j! \binom{n+1+r}{j+1} \genfrac{\{}{\}}{0pt}{}{k}{j}_{-r},
\end{equation}
from which, upon setting $n=0$, we further obtain
\begin{equation}\label{rah3}
(-1)^k S_k(r) = -\delta_{k,0} + \sum_{j=0}^{k} j! \binom{r+1}{j+1} \genfrac{\{}{\}}{0pt}{}{k}{j}_{-r}.
\end{equation}
Finally, substituting \eqref{rah3} in \eqref{rah2}, we get \eqref{th2}.

We conclude this section with the following two remarks.

\begin{remark}
By renaming $r$ as $n$ in \eqref{rah3} we find that
\begin{equation*}
S_k(n) = (-1)^k \Bigg( -\delta_{k,0} + \sum_{j=0}^{k} j! \binom{n+1}{j+1} \genfrac{\{}{\}}{0pt}{}{k}{j}_{-n} \Bigg),
\end{equation*}
which may be compared with \eqref{f1}.
\end{remark}

\begin{remark}
It is to be noted that the equation \eqref{rah1} above is equivalent to the equation appearing in \cite[Corollary 2.2]{rahm} in which $d=1$ and $a=r$.
\end{remark}

\section{Concluding remarks}

Let us observe that, by letting $r=n+m$ in \eqref{th1}, where $m$ is any given non-negative integer, and using \eqref{relt}, we obtain
\begin{equation}\label{con1}
S_k(n) = \sum_{j=0}^{k} (-1)^j j! \left[\binom{n+m+j-1}{j+1} - \binom{m+j-1}{j+1} \right]
\genfrac{\{}{\}}{0pt}{}{k}{j}_{n+m}.
\end{equation}
Of course, \eqref{con1} reduces to \eqref{reqn} and \eqref{reqn1} when $m=0$ and $m=1$, respectively. Similarly, putting $r=n-m$ in \eqref{th1}, where $m$ is any given non-negative integer, we obtain
\begin{equation}\label{con2}
S_k(n) = \sum_{j=0}^{k} j! \left[\binom{m+1}{j+1} + (-1)^j \binom{n+j-m-1}{j+1} \right]
\genfrac{\{}{\}}{0pt}{}{k}{j}_{n-m}.
\end{equation}
Note that, when $m =n$, \eqref{con2} reduces to \eqref{f1}.

On the other hand, applying the formula $\genfrac{\{}{\}}{0pt}{}{k}{j}_{-r} = (-1)^{k-j} \genfrac{\{}{\}}{0pt}{}{k}{j}_{r-j}$ (see \cite[Equation (2.4)]{ma}) and taking $r =k$ in \eqref{th2} yields
\begin{equation}\label{con3}
S_k(n) = \sum_{j=0}^{k} (-1)^{k-j} j! \left[\binom{n+k+1}{j+1} - \binom{k+1}{j+1} \right]
\genfrac{\{}{\}}{0pt}{}{k}{j}_{k-j}.
\end{equation}
Incidentally, for $n=1$, \eqref{con3} gives us the identity
\begin{equation*}
\sum_{j=0}^{k} (-1)^{k-j} j! \binom{k+1}{j} \genfrac{\{}{\}}{0pt}{}{k}{j}_{k-j} =1.
\end{equation*}

Moreover, from \eqref{pf1} and \eqref{rminus}, we obtain the following formula for the Bernoulli polynomials evaluated at the non-negative integer $r$
\begin{equation*}
B_{k+1}(r) = B_{k+1}(1) + (k+1) \sum_{j=0}^k (-1)^j j! \binom{r+j-1}{j+1} \genfrac{\{}{\}}{0pt}{}{k}{j}_r,
\quad \text{integer} \,\, r \geq 0.
\end{equation*}
Likewise, making $r \to -r$ in the preceding equation and using \eqref{relt} gives the following formula for the Bernoulli polynomials evaluated at the negative integer $-r$
\begin{equation*}
B_{k+1}(-r) = B_{k+1}(1) - (k+1) \sum_{j=0}^k j! \binom{r+1}{j+1} \genfrac{\{}{\}}{0pt}{}{k}{j}_{-r},
\quad \text{integer} \,\, r \geq 0.
\end{equation*}
One can naturally extend the above formula for $B_{k+1}(r)$ to apply to any real or complex variable $x$ as follows
\begin{equation*}
B_{k+1}(x) = B_{k+1}(1) + (k+1) \sum_{j=0}^k (-1)^j j! \binom{x+j-1}{j+1} \genfrac{\{}{\}}{0pt}{}{k}{j}_x,
\end{equation*}
where, using the notation in \cite{matsu}, $\genfrac{\{}{\}}{0pt}{}{k}{j}_x$ refers to the Stirling polynomial of the second kind $R_{k,j}(x)$, namely
\begin{equation*}
\genfrac{\{}{\}}{0pt}{}{k}{j}_x = \frac{1}{j!} \sum_{i=0}^{j} (-1)^{j-i} \binom{j}{i} (i+x)^k.
\end{equation*}

Finally, we point out that $B_{k}(x-1)$ can be expressed in the form
\begin{equation}\label{ber1}
B_{k}(x-1) = \sum_{j=0}^k (-1)^j j! H_{j+1} \genfrac{\{}{\}}{0pt}{}{k}{j}_{x},
\end{equation}
where $H_j = 1 + \frac{1}{2} + \cdots + \frac{1}{j}$ is the $j$-th harmonic number. In particular, setting $x=1$ in \eqref{ber1} yields the following known formula for the Bernoulli numbers (see, e.g., \cite[Equation (5.9)]{wang})
\begin{equation*}
B_{k} = \sum_{j=0}^k (-1)^j j! H_{j+1} \genfrac{\{}{\}}{0pt}{}{k+1}{j+1}.
\end{equation*}
Furthermore, from \eqref{pf1} and \eqref{ber1}, we arrive at the following formula for the sum of powers of integers
\begin{equation*}
S_{k-1}(n) = \frac{1}{k} \sum_{j=0}^k (-1)^j j! H_{j+1} \left( \genfrac{\{}{\}}{0pt}{}{k}{j}_{n+2}
\! - \genfrac{\{}{\}}{0pt}{}{k}{j}_{2} \right),
\end{equation*}
which holds for any integers $n \geq 0$ and $k \geq 1$, with $S_{k-1}(0) =0$ for all $k \geq 1$.


\begin{thebibliography}{99}


\bibitem{matsu} B\'{e}nyi, B., Matsusaka, T. (2021). Combinatorial aspects of poly-Bernoulli polynomials and poly-Euler numbers, preprint. Available at: arXiv:2106.05585v2 [math.CO].
\vspace{-2mm}

\bibitem{rahm} Bounebirat, F., Laissaoui, D., Rahmani, M. (2018). Several explicit formulae of sums and hyper-sums of powers of integers. {\it Online J. Anal. Comb.} 13, Article \#4, 9 pp.
\vspace{-2mm}

\bibitem{boya} Boyadzhiev, K. N. (2012). Close encounters with the Stirling numbers of the second kind. {\it Math. Mag.}
85(4):252--266.
\vspace{-2mm}

\bibitem{boya2} Boyadzhiev, K. N. (2020). Sums of powers and special polynomials. {\it Discuss. Math. Gen. Algebra Appl.} 40(2):275--283.
\vspace{-2mm}

\bibitem{broder} Broder, A. Z. (1984). The $r$-Stirling numbers. {\it Discrete Math.} 49(3):241--259.
\vspace{-2mm}

\bibitem{carlitz} Carlitz, L. (1980). Weighted Stirling numbers of the first and second kind---I. {\it Fibonacci Quart.} 18(2):147--162.
\vspace{-2mm}

\bibitem{cere} Cereceda, J. L. (2015). Newton's interpolation polynomial for the sums of powers of integers. {\it Amer. Math. Monthly.} 122(10):1007.
\vspace{-2mm}

\bibitem{chrysafi} Chrysafi, L., Marques, C. A. (2021). Sums of powers via matrices. {\it Math. Mag.} 94(1):43--52.
\vspace{-2mm}

\bibitem{atlan} El-Mikkawy, M., Atlan, F. (2013). Derivation of identities involving some special polynomials and numbers via generating functions with applications. {\it Appl. Math. Comput.} 220:518--535.
\vspace{-2mm}

\bibitem{ginsburg} Ginsburg, J. (1928). Note on Stirling's numbers. {\it Amer. Math. Monthly.} 35(2):77--80.
\vspace{-2mm}

\bibitem{kargin} Karg{\i}n, L., Dil, A., Can, M. (2020). Formulas for sums of powers of integers and their reciprocals, preprint. Available at: arXiv:2006.01132v1 [math.CO].
\vspace{-2mm}

\bibitem{kargin2} Karg{\i}n, L., Cenkci, M., Dil, A., Can, M. (2022). Generalized harmonic numbers via poly-Bernoulli polynomials. {\it Publ. Math. Debrecen.} 100(3-4):365--386.
\vspace{-2mm}

\bibitem{koutras} Koutras, M. (1982). Non-central Stirling numbers and some applications. {\it Discrete Math.} 42(1):73--89.
\vspace{-2mm}

\bibitem{ma} Ma, Q., Wang, W. (2023). Riordan arrays and $r$-Stirling number identities. {\it Discrete Math.} 346(1):113211.
\vspace{-2mm}

\bibitem{merris} Merris, R. (2000). The $p$-Stirling numbers. {\it Turk. J. Math.} 24(4):379--399.
\vspace{-2mm}

\bibitem{newsome} Newsome, N. J., Nogin, M. S., Sabuwala, A. H. (2017). A proof of symmetry of the power sum polynomials using a novel Bernoulli number identity. {\it J. Integer Seq.} 20, Article 17.6.6, 10 pp.
\vspace{-2mm}

\bibitem{gould} Quaintance, J., Gould, H. W. (2016). {\it Combinatorial Identities for Stirling Numbers. The Unpublished Notes of H W Gould}. Singapore: World Scientific Publishing.
\vspace{-2mm}

\bibitem{shirali} Shirali, S. (2018). Stirling set numbers \& powers of integers. {\it At Right Angles.} 7(1):26--32.
\vspace{-2mm}

\bibitem{wang} Wang, W. (2010). Riordan arrays and harmonic number identities. {\it Comput. Math. Appl.} 60(5):1494--1509.
\vspace{-2mm}

\bibitem{witula} Witu{\l}a, R., Kaczmarek, K., Lorenc, P., Hetmaniok, E., Pleszczy\'{n}ski. (2014). Jordan numbers, Stirling numbers and sums of powers. {\it Discuss. Math. Gen. Algebra Appl.} 34(2):155--166.


\end{thebibliography}
\end{document}